\documentclass[11pt,reqno]{amsart}

\usepackage[utf8]{inputenc}
\usepackage[T1]{fontenc}
\usepackage{amsmath,amssymb,amsthm}
\usepackage{mathtools}
\usepackage{enumitem}
\usepackage[colorlinks=true,linkcolor=blue,citecolor=blue,urlcolor=blue]{hyperref}

\theoremstyle{plain}
\newtheorem{theorem}{Theorem}
\newtheorem{lemma}[theorem]{Lemma}

\theoremstyle{remark}
\newtheorem{remark}[theorem]{Remark}

\newcommand{\Z}{\mathbb{Z}}
\newcommand{\C}{\mathbb{C}}

\begin{document}

\title[Conway polynomials of $Th(p,q)$ and $Th(q,p)$]
      {Conway polynomials of $Th(p,q)$ and $Th(q,p)$:\\ a proof of Takemura's conjecture}

\author{Werner Alois Stanggassinger}
\address{Wasserburg am Inn, Bavaria, Germany}
\email{stanggassinger@stainlesssteel4u.de}

\thanks{This work was developed in collaboration with the AI system Claude Code (Anthropic):
the choice of problem, the direction of the work and the computing hardware are due to the author
(a sculptor); the formalisation, the algorithms and the proofs were developed by Claude Code
(Anthropic) working with the author. The collaboration is stated as a fact, not a disclaimer; details
are in the section on tool use at the end.}


\date{19 October 2026 (first version dated 5 October 2026, on Zenodo since 6 October 2026)}

\begin{abstract}
Takemura conjectured that the Conway polynomials $\sum a_\ell z^\ell$ of the Turk's head link $Th(p,q)$ and
$\sum b_\ell z^\ell$ of $Th(q,p)$ satisfy $a_\ell=b_\ell$ for $\ell\equiv0,1\pmod 4$ and $a_\ell=-b_\ell$ for $\ell\equiv2,3\pmod4$, and
proved it for $q=2,3$. We prove it for all $p,q\ge1$, knots and links. Equivalently,
$\nabla_{Th(q,p)}(z)=c\,\nabla_{Th(p,q)}(iz)$ with $c=1$ if $\gcd(p,q)$ is odd and $c=-i$ if it is even. The
proof combines a product formula $\nabla_{Th(p,q)}(z)^2=\prod_{j=1}^{p-1}\prod_{k=1}^{q-1}\big(z^2+2\cos\frac{2\pi k}{q}-2\cos\frac{2\pi j}{p}\big)$,
which comes from the reduced Burau representation, with the leading coefficients, read off the Seifert surface of the closed braid.
\end{abstract}

\maketitle

\section{Result}
\label{sec:result}

For $p\ge1$ let $A_p=\sigma_1\sigma_2^{-1}\sigma_3\sigma_4^{-1}\cdots\sigma_{p-1}^{\pm1}\in B_p$ ($A_1$ is the
trivial braid). The Turk's head link $Th(p,q)$ is the closure of $A_p^{\,q}$; it has $\mu=\gcd(p,q)$
components \cite[\S2]{DS}, and we orient it as a closed braid, all strands running in the same direction. For an oriented link $L$ let $\nabla_L(z)\in\Z[z]$ be its Conway polynomial
($\nabla_{\mathrm{unknot}}=1$, $\nabla_{L_+}-\nabla_{L_-}=z\nabla_{L_0}$); only powers $z^\ell$ with
$\ell\equiv\mu-1\pmod 2$ occur. We write $\Delta_L(t)=\nabla_L(t^{1/2}-t^{-1/2})$ for the Conway-normalised
Alexander polynomial.

Takemura \cite[Thm.~1.1]{Tak16} proved for $m\ge1$ and $n=2,3$ that, writing
$\nabla_{Th(m,n)}=\sum a_\ell z^\ell$ and $\nabla_{Th(n,m)}=\sum b_\ell z^\ell$,
\begin{equation}
\label{eq:rule}
a_\ell=b_\ell\ \text{ for } \ell\equiv0,1\pmod4,\qquad a_\ell=-b_\ell\ \text{ for } \ell\equiv2,3\pmod4,
\end{equation}
and conjectured \eqref{eq:rule} for all $m\ge1$, $n\ge2$ \cite[Conj.~5.1]{Tak16}. Di Prisa and \c{S}avk record the
conjecture next to their Problem~E \cite[\S6.4]{DS}.

\begin{theorem}
\label{thm:T}
\eqref{eq:rule} holds for all $p,q\ge1$, i.e.\ for $\nabla_{Th(p,q)}=\sum a_\ell z^\ell$ and $\nabla_{Th(q,p)}=\sum b_\ell z^\ell$.
\end{theorem}

Here $i=\sqrt{-1}$. Since $i^\ell=(-1)^{\ell/2}$ for even $\ell$ and $-i\cdot i^\ell=(-1)^{(\ell-1)/2}$ for odd $\ell$, and since
$a_\ell=b_\ell=0$ unless $\ell\equiv\mu-1\pmod2$, Theorem~\ref{thm:T} is equivalent to
\begin{equation}
\label{eq:iz}
\nabla_{Th(q,p)}(z)=c_\mu\,\nabla_{Th(p,q)}(iz),\qquad c_\mu=\begin{cases}1,&\mu\ \text{odd},\\-i,&\mu\ \text{even}.\end{cases}
\end{equation}
Mirroring both links replaces $\nabla(z)$ by $\nabla(-z)$ on both sides, so \eqref{eq:rule} and \eqref{eq:iz}
do not depend on the handedness convention for $A_p$. If $p=1$ or $q=1$ both links are unknots and
\eqref{eq:rule} is trivial; we assume $p,q\ge2$. For $p=q$, \eqref{eq:rule} says that $a_\ell=0$ for all
$\ell\equiv2,3\pmod4$.

\begin{remark}[What we have read]
\label{rem:read}
We have read \cite{Tak16} (in its accepted manuscript, Kobe University Repository,
\href{https://hdl.handle.net/20.500.14094/90003785}{hdl:20.500.14094/90003785}), \cite{DS} and Takemura's doctoral thesis \cite{TakD} (cited in \cite{Sa2} as
``Tech.\ Report 7116''). The thesis proves the cases $n=2,3$, states ``We can not prove this Theorem for $n\ge4$
yet'', and repeats the conjecture; its Section~6 gives a Seifert matrix of $Th(m,n)$ of the block form described in
Section~\ref{sec:lc} \cite[Lemma~6.1]{TakD}. Takemura's article \cite{Tak18} is quoted in \cite[\S4.6]{DS} for a
recursive determinant formula for $\Delta_{Th(p,q)}$ and a divisibility theorem; we have not read it. A search of
Semantic Scholar on 5 October 2026 found \cite{DS} and \cite{Tak18} as the only works citing \cite{Tak16}, and
\cite{DS} records it as a conjecture, in its section on open problems (\S6.4).

What is known about the product formula: Blackwell, Das, Mayer, Moyar, Quraishi and Stees \cite[Thm.~3.5]{BDMMQS}
write the Alexander polynomial of a spiral knot, and by their Section~5 of a spiral link, as a product over the
non-trivial $q$-th roots of unity of a recursively defined polynomial $C_{p-1}$; the closed form of the factors used
below is \cite[Thm.~2]{H}.
\end{remark}

\section{The product formula}
\label{sec:product}

Put $u_k=2\cos(2\pi k/q)$, $v_j=2\cos(2\pi j/p)$ and
\[
P_{p,q}(z):=\prod_{j=1}^{p-1}\prod_{k=1}^{q-1}\big(z^2+u_k-v_j\big).
\]
Its coefficients are rational integers: they are algebraic integers, and the Galois group of $\mathbb{Q}(e^{2\pi i/pq})$ permutes the $u_k$ among themselves and the $v_j$ among themselves.

\begin{lemma}
\label{lem:P}
For $p,q\ge2$, $\nabla_{Th(p,q)}(z)^2=P_{p,q}(z)$.
\end{lemma}

The proof uses two cited facts.

\smallskip
\noindent\textbf{(B) Burau's theorem.} For a braid $\beta\in B_p$ and $S$ the reduced Burau matrix, obtained from
Birman's matrices (3-26) by deleting the last row and column \cite[p.~122]{Bir},
$(1+t+\dots+t^{p-1})\,\tilde\nabla_{\hat\beta}(t)=\det\big(S(\beta)-I\big)$ \cite[Thm.~3.11, (3-28)]{Bir}, where
$\tilde\nabla$ is the reduced (one-variable) Alexander polynomial of the closure; by Birman's footnote~7 on p.~120
this holds for links of any number of components.

\smallskip
\noindent\textbf{(S) Two standard facts}, used without re-proof: (S1) $\tilde\nabla_L(t)$ agrees, up to a unit
$\pm t^{k}$, $k\in\tfrac12\Z$, with $\det(t^{1/2}V-t^{-1/2}V^{\mathsf T})$ for any Seifert matrix $V$ of $L$;
(S2) there is a sign $\eta\in\{\pm1\}$, fixed by the conventions and the same for all oriented links, such that
$\nabla_L(t^{1/2}-t^{-1/2})=\eta^{\,n}\det(t^{1/2}V-t^{-1/2}V^{\mathsf T})$ for every Seifert matrix $V$ of $L$ of size $n$.
(SageMath computes $\nabla$ through a Seifert matrix; the checks in Section~\ref{sec:checks} agree with (S1) and (S2).)

\smallskip
\noindent\textbf{(F) The factorisation of \cite{H}.} With $m=\lfloor(p-1)/2\rfloor$ and $c_j=4\cos^2(\pi j/p)=2+v_j$,
\[
\det\big(I-x\,S(A_p)\big)=(1+xt)^{[p\ \mathrm{even}]}\;t^{-m}\prod_{j=1}^{m}\big(x\,t^{2}+(x^{2}-c_jx+1)t+x\big)
\]
for $p\ge3$ \cite[Thm.~2]{H}; for $p=2$, $S(A_2)=(-t)$ and the same formula holds with $m=0$.

\begin{proof}[Proof of Lemma~\ref{lem:P}]
Let $\zeta=e^{2\pi i/q}$ and $z=t^{1/2}-t^{-1/2}$, so $t+t^{-1}=z^2+2$. Since
$\det(I-S^q)=\prod_{k=0}^{q-1}\det(I-\zeta^kS)$ and the $k=0$ factor is $t^{-m}(1+t+\dots+t^{p-1})$
\cite[proof of Cor.~4]{H}, (B), (S1), (S2) and (F) give, up to a unit $\pm t^{\nu/2}$, $\nu\in\Z$ (the scalars $\zeta^k$ below multiply to $\pm1$ over $1\le k\le q-1$),
\[
\nabla_{Th(p,q)}\big(t^{1/2}-t^{-1/2}\big)\;\doteq\;\prod_{k=1}^{q-1}\Big[(1+\zeta^kt)^{[p\ \mathrm{even}]}\prod_{j=1}^{m}Q_j(t,\zeta^k)\Big],
\]
where $Q_j(t,x)=x\,t^{2}+(x^{2}-c_jx+1)t+x$.
For $x=\zeta^k$ we have $x+x^{-1}=u_k$, hence $Q_j(t,x)=xt\,(t+t^{-1}+u_k-c_j)=xt\,(z^2+u_k-v_j)$. For even $p$ the
linear factors pair up: for $k\ne q-k$, $(1+\zeta^kt)(1+\zeta^{-k}t)=t\,(t+t^{-1}+u_k)=t\,(z^2+u_k-v_{p/2})$, because
$v_{p/2}=-2$; for even $q$ the unpaired factor $k=q/2$ is $1-t=-t^{1/2}z$. So $\nabla_{Th(p,q)}(z)\doteq G(z)$ with
\[
G(z)=\prod_{k=1}^{q-1}\prod_{j=1}^{m}\big(z^2+u_k-v_j\big)\cdot
\Big[\prod_{k=1}^{\lfloor (q-1)/2\rfloor}\big(z^2+u_k-v_{p/2}\big)\Big]^{[p\ \mathrm{even}]}\cdot z^{[p,q\ \mathrm{even}]}.
\]
Now compare $G^2$ with $P_{p,q}$. Since $v_j=v_{p-j}$, the indices $j$ and $p-j$ give equal factors, and $j=p/2$
(for even $p$) is the only fixed point; likewise $u_k=u_{q-k}$ with fixed point $k=q/2$ (for even $q$). Hence
\[
P_{p,q}=\prod_{k=1}^{q-1}\prod_{j=1}^{m}(z^2+u_k-v_j)^2\cdot\Big[\prod_{k=1}^{q-1}(z^2+u_k-v_{p/2})\Big]^{[p\ \mathrm{even}]},
\]
and the last product equals $\prod_{k\le\lfloor(q-1)/2\rfloor}(z^2+u_k-v_{p/2})^2$, times $z^2+u_{q/2}-v_{p/2}=z^2$ when
$q$ is also even. Thus $P_{p,q}=G^2$.

It remains to remove the unit. Write $s=t^{1/2}$; then $\nabla(s-s^{-1})=\pm s^{k}G(s-s^{-1})$ for some $k\in\Z$.
Both $\nabla(s-s^{-1})$ and $G(s-s^{-1})$ are invariant under $s\mapsto-s^{-1}$, so $s^k=(-1)^ks^{-k}$, which forces
$k=0$ because $G\ne0$. Hence $\nabla_{Th(p,q)}=\pm G$ and $\nabla_{Th(p,q)}^2=P_{p,q}$.
\end{proof}

\begin{remark}[Determinants]
\label{rem:det}
At $t=-1$, i.e.\ $z=2i$, $|\nabla_L(2i)|=|\Delta_L(-1)|=\det L$, so Lemma~\ref{lem:P} gives
$\det Th(p,q)=|P_{p,q}(2i)|^{1/2}$, where $-4+u_k-v_j=-\big(4\sin^2\frac{\pi k}{q}+4\cos^2\frac{\pi j}{p}\big)$. Let $p$ be even.
The factors with $j$ and $p-j$ are equal since $v_j=v_{p-j}$; we write $\cos^2\frac{\pi j}{p}=\sin^2\frac{\pi(p/2-j)}{p}$, and the factors with $j=p/2$
contribute $\prod_{k=1}^{q-1}4\sin^2\frac{\pi k}{q}=q^2$ to $|P_{p,q}(2i)|$; hence
\[
\det Th(p,q)=q\prod_{k=1}^{q-1}\prod_{j=1}^{p/2-1}\Big(4\sin^2\frac{\pi k}{q}+4\sin^2\frac{\pi j}{p}\Big).
\]
This is the
determinant formula of Dowdall, Mattman, Meek and Solis for even $p\ge4$, $q\ge3$, as recorded in \cite[Thm.~H]{DS}.
\end{remark}

\section{Leading coefficients}
\label{sec:lc}

We read the leading coefficient off the Seifert surface of the closed braid $\widehat{A_p^{\,q}}$, so that the
orientation conventions are the same for $Th(p,q)$ and $Th(q,p)$. Write $A_p^{\,q}$ with the letters
$\sigma_i^{s_i}$, $s_i=(-1)^{i+1}$, and call $\sigma_1$ a positive crossing (the opposite convention is discussed
after Lemma~\ref{lem:lc}).

\smallskip
\noindent\emph{The surface.} The Seifert circles of the closed braid are $p$ concentric circles, all oriented the
same way. Let $F$ consist of discs $D_1,\dots,D_p$ bounded by them, $D_i$ at height $h=i$, and of one half-twisted
band for each crossing; the $q$ bands $B_{i,1},\dots,B_{i,q}$ of the letters $\sigma_i^{s_i}$ join $D_i$ to $D_{i+1}$
inside the slab $i\le h\le i+1$. Then $F$ is a connected Seifert surface with $\chi(F)=p-q(p-1)$, so
$H_1(F)$ has rank $N=(p-1)(q-1)$. For each $i$, $F_i=D_i\cup D_{i+1}\cup\bigcup_k B_{i,k}$ is the surface of Seifert's
algorithm for the closure of the $2$-braid $\sigma_1^{s_iq}$, and $H_1(F)$ has the basis $\gamma_{i,k}$
($1\le i\le p-1$, $1\le k\le q-1$), where $\gamma_{i,k}\subset F_i$ runs through $B_{i,k}$ and $B_{i,k+1}$.

\smallskip
\noindent\emph{The Seifert matrix is block triangular.} Let $V(a,b)=\operatorname{lk}(a,b^+)$, $b^+$ the push-off of $b$
to the positive side of $F$, and order the basis by $i$. For $j\ge i+1$ the curves $\gamma_{i,k}\subset\{h\le i+1\}$
and $\gamma_{j,l}\subset\{h\ge i+1\}$ meet the plane $h=i+1$ at most in $D_{i+1}$, where $F$ is horizontal. If the positive
normal of the discs points up, the push-off $\gamma_{j,l}^+$ can be taken in $\{h>i+1\}$, so $\gamma_{i,k}$ and
$\gamma_{j,l}^+$ are separated by a plane and $V(\gamma_{i,k},\gamma_{j,l})=0$; if it points down, the same argument gives
$V(\gamma_{j,l},\gamma_{i,k})=0$. Either way $V$ is block triangular, and its diagonal block $V_i$ is the Seifert matrix
of $F_i$ in the basis $\gamma_{i,\cdot}$ (the linking numbers of curves on $F_i$ and their push-offs do not see the rest of $F$).
This is the block form of the Seifert matrices in \cite[Lemma~6.1]{TakD} and \cite[\S4.6]{DS}.

\begin{lemma}
\label{lem:lc}
With $\sigma_1$ a positive crossing, $\nabla_{Th(p,q)}$ has degree $N=(p-1)(q-1)$ and leading coefficient
$(-1)^{(q-1)\lfloor (p-1)/2\rfloor}$.
\end{lemma}

\begin{proof}
By (S2), for a Seifert matrix $V$ of size $n$ the polynomial $\nabla_L$ has degree at most $n$, and its coefficient of
$z^n$ is $\eta^n\det V$: the coefficient of $t^{n/2}$ in $\det(t^{1/2}V-t^{-1/2}V^{\mathsf T})$ is $\det V$, and $z^n$
is the only power of $z=t^{1/2}-t^{-1/2}$ containing $t^{n/2}$, with coefficient $1$.

For the closure $L_n$ of $\sigma_1^{n}$ ($n\ge0$) the skein relation at one crossing gives
$\nabla_{L_n}=z\nabla_{L_{n-1}}+\nabla_{L_{n-2}}$, with $\nabla_{L_0}=0$ (two-component unlink) and $\nabla_{L_1}=1$; so
$\nabla_{L_n}$ has degree $n-1$ and leading coefficient $1$. For the closure $L_{-n}$ of $\sigma_1^{-n}$ likewise
$\nabla_{L_{-n}}=-z\nabla_{L_{-(n-1)}}+\nabla_{L_{-(n-2)}}$, leading coefficient $(-1)^{n-1}$. Since $\partial F_i=L_{s_iq}$
and $V_i$ has size $q-1$, this gives $\eta^{q-1}\det V_i=1$ for $s_i=1$ and $=(-1)^{q-1}$ for $s_i=-1$; in particular
$\det V_i\ne0$. Hence $\det V=\prod_i\det V_i\ne0$, so $\nabla_{Th(p,q)}$ has degree exactly $N$, and its leading coefficient is
\[
\eta^N\det V=\prod_{i=1}^{p-1}\eta^{q-1}\det V_i=\prod_{i:\,s_i=-1}(-1)^{q-1}=(-1)^{(q-1)\lfloor(p-1)/2\rfloor},
\]
as $s_i=-1$ exactly for the $\lfloor(p-1)/2\rfloor$ even indices $i\le p-1$. The unknown sign $\eta$ has cancelled.
\end{proof}

With the opposite convention ($\sigma_1$ negative) every $s_i$ changes sign, and the exponent becomes
$(q-1)\lceil(p-1)/2\rceil$; this is the mirror image, and by Section~\ref{sec:result} the theorem does not depend on it.
The value of $c$ below is also the same with either exponent.

\section{Proof of Theorem~\ref{thm:T}}
\label{sec:proof}

Let $p,q\ge2$, $\mu=\gcd(p,q)$, $N=(p-1)(q-1)$. Interchanging the roles of $j$ and $k$,
\[
P_{q,p}(z)=\prod_{k=1}^{q-1}\prod_{j=1}^{p-1}\big(z^2-u_k+v_j\big)=(-1)^N\prod_{k,j}\big((iz)^2+u_k-v_j\big)=(-1)^NP_{p,q}(iz).
\]
By Lemma~\ref{lem:P}, $\nabla_{Th(q,p)}(z)^2=(-1)^N\nabla_{Th(p,q)}(iz)^2$ in $\C[z]$, so
$\nabla_{Th(q,p)}(z)=c\,\nabla_{Th(p,q)}(iz)$ for a constant $c$ with $c^2=(-1)^N$. Comparing coefficients of $z^N$ by Lemma~\ref{lem:lc} (both leading coefficients are $\pm1$),
\[
c=(-1)^{E}\;i^{-N},\qquad E=(p-1)\lfloor (q-1)/2\rfloor+(q-1)\lfloor(p-1)/2\rfloor .
\]
\emph{$p,q$ odd.} $E=(p-1)(q-1)$ is even, and $N\equiv0\pmod4$; so $c=1=c_\mu$.

\emph{$p$ even, $q$ odd} (the case $p$ odd, $q$ even is symmetric). $E=(p-1)\tfrac{q-1}{2}+(q-1)\tfrac{p-2}{2}
=\tfrac{q-1}{2}(2p-3)\equiv\tfrac{q-1}2\pmod2$, and since $q-1$ is even and $p-1$ odd,
$i^{-N}=\big(i^{q-1}\big)^{-(p-1)}=\big((-1)^{(q-1)/2}\big)^{-(p-1)}=(-1)^{(q-1)/2}$. So $c=1=c_\mu$.

\emph{$p=2a$, $q=2b$.} $E=(2a-1)(b-1)+(2b-1)(a-1)\equiv a+b\pmod2$, and $N=4ab-2a-2b+1$, so
$i^{-N}=-i\,(-1)^{a+b}$. So $c=-i=c_\mu$.

(With $\lceil\cdot\rceil$ in place of $\lfloor\cdot\rfloor$, $E$ changes by $[p\text{ even}](q-1)+[q\text{ even}](p-1)$,
which is even in every case.)

In all cases \eqref{eq:iz} holds, which is Theorem~\ref{thm:T}. \qed

\begin{remark}[Degree for links]
\label{rem:degree}
For links the degree of $\Delta_{Th(p,q)}$ (span in $t$) is $N=(p-1)(q-1)=2g_3+\mu-1$ (Lemma~\ref{lem:lc}); the
expression $(p-1)(q-1)+1-\gcd(p,q)$ in \cite[\S4.6(2)]{DS} is $2g_3$ \cite[Thm.~E]{DS} and agrees with it exactly for
knots, where $\mu=1$; for links the span is $2g_3+\mu-1$. For instance $\nabla_{Th(2,2)}=\pm z$ has degree $1$.
\end{remark}

\section{Checks}
\label{sec:checks}

The proof does not depend on them. Conway polynomials were computed with SageMath's
\texttt{Link.conway\_polynomial()} (Seifert-matrix route, independent of Burau) for the closure of $A_p^{\,q}$.
\begin{itemize}[leftmargin=1.4em]
\item \eqref{eq:rule} holds for all $28$ pairs $2\le p\le q\le 8$ (knots and links); the variant without the sign
  change fails in all $21$ pairs with $p\ne q$.
\item All $29$ Conway polynomials listed in \cite[\S1 and \S5]{Tak16} are reproduced: $18$ exactly, and the
  $11$ with even $\mu$ up to $z\mapsto-z$, i.e.\ our $A_p$ is the mirror of Takemura's convention. This is visible only
  for even $\mu$: for odd $\mu$ the link has an odd number $\mu$ of components, and its Conway polynomial contains only even
  powers of $z$.
\item \eqref{eq:iz} holds exactly (coefficients in $\mathbb{Q}(i)$) for all $81$ pairs $2\le p,q\le10$, and the degree
  and leading coefficient of Lemma~\ref{lem:lc} hold in all $81$; with $c=+i$ instead of $-i$ the identity fails in all
  $25$ pairs with $p,q$ even.
\item Lemma~\ref{lem:P} exactly (coefficients in cyclotomic fields) for all $64$ pairs $2\le p,q\le9$; the variant with factors $z^2-u_k+v_j$ fails in all $56$ pairs with $p\ne q$.
\item The block matrix $V$ of \cite[\S4.6]{DS} (diagonal blocks alternately $A_q$ and $-A_q^{\mathsf T}$, below them $-A_q$,
  $A_q$ with $-1$ on the diagonal and $1$ above it) for $2\le p,q\le8$ ($49$ pairs): $\det(t^{1/2}V-t^{-1/2}V^{\mathsf T})$
  equals SageMath's Conway polynomial exactly for odd $\mu$ ($33$) and up to $z\mapsto-z$ for even $\mu$ ($16$);
  $\det V=(-1)^{(q-1)\lceil(p-1)/2\rceil}$, and this is the leading coefficient of that polynomial, in all $49$
  (in this matrix convention the exponent has $\lceil\cdot\rceil$; cf.\ the remark after Lemma~\ref{lem:lc}); with
  all diagonal blocks equal to $A_q$ (no $-A_q^{\mathsf T}$) the result is wrong in all pairs with $p\ge3$ ($42$ of the $49$).
\item The value of $c$ in Section~\ref{sec:proof} equals $c_\mu$ for all $2\le p,q\le400$ ($159\,201$ pairs), with
  $\lfloor\cdot\rfloor$ and with $\lceil\cdot\rceil$.
\end{itemize}
The scripts and their outputs are deposited at
\href{https://doi.org/10.5281/zenodo.23183920}{doi:10.5281/zenodo.23183920}.

\subsection*{Tool use}
This work was prepared with substantial use of Claude Code (Anthropic), an agentic AI system based on
large language models, working together with the author.
\begin{itemize}[leftmargin=1.4em]
\item \emph{Mathematics.} Takemura's conjecture was among the open problems found in a literature search
carried out with Claude Code, in the survey~\cite[\S6.4]{DS}; the author decided that this note be written.
The proof was worked out by Claude Code under the author's direction.
\item \emph{Writing and typesetting.} The text and the \LaTeX{} source were written by Claude Code, in
English. The note contains no figures.
\item \emph{References.} Literature was searched and sourced with Claude Code. Quotations and statements
from \cite{Tak16}, \cite{TakD}, \cite{DS}, \cite{BDMMQS}, \cite{Sa2} and \cite{Bir} were checked against these
texts (for \cite{Bir}, against scans of the cited pages); \cite{Tak18} was not read and is cited
as in~\cite{DS}. The bibliographic data of \cite{Tak16} were checked against Crossref.
\item \emph{Code.} The checks of Section~\ref{sec:checks} were written and run by Claude Code in Python,
with SageMath~10.9 (Python~3.12.14). The deposited scripts, at
\href{https://doi.org/10.5281/zenodo.23183920}{doi:10.5281/zenodo.23183920}, are English versions of the working scripts, prepared by
Claude Code and run again before deposit, with the results stated in Section~\ref{sec:checks}.
\item \emph{Checking.} The draft was read in full by separate Claude Code agents that had not read
the working notes, twice on 5 October and once on 19 October 2026, and each later change was checked in a diff read; their findings were checked and incorporated where they held.
\end{itemize}
The author takes full responsibility for all contents.



\begin{thebibliography}{XXXXXX}

\bibitem[BDMMQS]{BDMMQS}
S.~Blackwell, A.~Das, S.~Mayer, L.~Moyar, F.~L.~Quraishi and R.~Stees, \emph{Classical invariants of spiral knots},
\href{https://arxiv.org/abs/2506.17889}{arXiv:2506.17889v2} (2026).

\bibitem[Bir74]{Bir}
J.~S.~Birman, \emph{Braids, Links, and Mapping Class Groups}, Annals of Mathematics Studies~82, Princeton
University Press, Princeton, New Jersey, copyright 1974.

\bibitem[DS]{DS}
A.~Di Prisa and O.~\c{S}avk, \emph{Turk's head knots and links: a survey},
\href{https://arxiv.org/abs/2409.20106}{arXiv:2409.20106v3} (2026); section and problem numbers refer to
this version. Journal version: Expo.\ Math.\ \textbf{44} (2026), 125777.

\bibitem[H]{H}
W.~A. Stanggassinger, \emph{Hoste's condition for all Turk's head knots}, note, Zenodo, first version 5 October 2026,
\href{https://doi.org/10.5281/zenodo.23154788}{doi:10.5281/zenodo.23154788}.

\bibitem[Sa]{Sa2}
S.~Saurabh, \emph{Spectral factorization and hypergeometric representations of the Alexander
polynomials of $Th(4,2n+1)$}, \href{https://arxiv.org/abs/2606.11301}{arXiv:2606.11301v1} (2026).


\bibitem[Tak16]{Tak16}
A.~Takemura, \emph{On a relation between Conway polynomials of $(m,n)$- and $(n,m)$-Turk's head links},
J.\ Knot Theory Ramifications \textbf{25} (2016), no.~14, 1650082.

\bibitem[TakD]{TakD}
A.~Takemura, \emph{Relations among Alexander--Conway polynomials of Turk's head links}, doctoral thesis, Kobe
University, degree conferred 25 March 2018, \href{https://hdl.handle.net/20.500.14094/D1007116}{hdl:20.500.14094/D1007116}.

\bibitem[Tak18]{Tak18}
A.~Takemura, \emph{On divisibility between Alexander polynomials of Turk's head links}, J.\ Knot Theory
Ramifications \textbf{27} (2018), no.~6, 1850040.

\end{thebibliography}
\end{document}